\documentclass[a4paper,12pt]{amsart}
\usepackage{amssymb,amsthm,amsmath,color,url}
\usepackage[margin=3cm]{geometry}

\usepackage{graphicx}
\usepackage{stmaryrd}
\usepackage[T1]{fontenc}
\usepackage[english]{babel}
\usepackage{comment}
\usepackage{hyperref}
\usepackage{ifthen}
\hypersetup{
  colorlinks   = true,
  linkcolor = red!70!black,
     citecolor    = green!50!black
}
\usepackage[noabbrev,capitalise]{cleveref}
\crefname{subsection}{Subsection}{Subsections}
\crefname{claim}{Claim}{Claims}
\crefname{problem}{Problem}{Problems}

\makeatletter
\def\namedlabel#1#2{\begingroup
   \def\@currentlabel{#2}%
   \label{#1}\endgroup
}
\makeatother

\usepackage{thmtools}
\usepackage{thm-restate}
\usepackage{enumitem}
\usepackage{caption,subcaption}

\usepackage{tikz}
\usetikzlibrary{shadows}
\usetikzlibrary{backgrounds}
\usetikzlibrary{arrows}
\usetikzlibrary{patterns}
\usetikzlibrary{calc}
\usetikzlibrary{shapes}
\usetikzlibrary{positioning}
\usetikzlibrary{math}
\usetikzlibrary{external}
\usetikzlibrary{decorations.pathreplacing}
\declaretheorem[name=Theorem, numberwithin=section]{theorem}
\declaretheorem[name=Lemma, sibling=theorem]{lemma}

\declaretheorem[name=Corollary, sibling=theorem]{corollary}
\declaretheorem[name=Conjecture, sibling=theorem]{conjecture}
\declaretheorem[name=Problem, sibling=theorem]{problem}
\declaretheorem[name=Question, sibling=theorem]{question}
\declaretheorem[name=Claim, sibling=theorem]{claim}
\declaretheorem[name=Claim, numbered=no]{claim*}

\newenvironment{proofofclaim}{\noindent{\emph{Proof of the Claim}:}}{\hfill$\Diamond$\medskip}

\declaretheorem[name=Observation, style=theorem, sibling=theorem]{observation}
\def\cqedsymbol{\ifmmode$\lrcorner$\else{\unskip\nobreak\hfil
\penalty50\hskip1em\null\nobreak\hfil$\lrcorner$
\parfillskip=0pt\finalhyphendemerits=0\endgraf}\fi}

\newcommand{\lcr}{\mathrm{lcr}}

\newcommand*\torso[1]{\llbracket #1 \rrbracket}
\newcommand{\Sep}{(Y,S,Z)}
\newcommand{\Kinf}{K_{\infty}}

\newcommand*\sg[1]{\{ #1 \}}

\let\le\leqslant
\let\ge\geqslant
\let\leq\leqslant
\let\geq\geqslant

\begin{document}

\title{Coarse geometry of  quasi-transitive graphs beyond planarity}

\author[L.~Esperet]{Louis Esperet}
\address[L.~Esperet]{Univ.\ Grenoble Alpes, CNRS, Laboratoire G-SCOP,
  Grenoble, France}
\email{louis.esperet@grenoble-inp.fr}

\author[U.~Giocanti]{Ugo Giocanti}
\address[U.~Giocanti]{Univ.\ Grenoble Alpes, CNRS, Laboratoire G-SCOP,
  Grenoble, France}
\email{ugo.giocanti@grenoble-inp.fr}

\thanks{The authors are partially supported by the French ANR Project
  GrR (ANR-18-CE40-0032), TWIN-WIDTH
  (ANR-21-CE48-0014-01), and by LabEx
  PERSYVAL-lab (ANR-11-LABX-0025).}

\date{}

\begin{abstract}
  We study geometric and topological properties of infinite graphs that are
  quasi-isometric to a planar graph of bounded degree. 
  We prove that every locally finite quasi-transitive graph excluding a minor is
  quasi-isometric to a planar graph of bounded degree. We use the
  result to give a simple proof of the result that finitely generated
  minor-excluded groups have Assouad-Nagata dimension at most 2
  (this is known to hold in greater generality, but all known
  proofs use significantly deeper tools). We also prove that every locally finite quasi-transitive graph that is
  quasi-isometric to a planar graph is $k$-planar
  for some $k$ (i.e.\ it has a planar drawing with at most $k$
  crossings per edge), and discuss a possible approach to prove the converse statement.
\end{abstract}

\maketitle

\section{Introduction}

Our work is motivated by a conjecture and a problem raised recently
by Georgakopoulos and Papasoglu \cite{GP23}, lying at the
intersection of metric graph theory and graph minor theory. Before we
state them, we first
need to introduce some terminology.

\smallskip

We say
that a graph
$H$ is a \emph{$k$-fat minor} of a graph $G$ if there exists a family
of connected subsets $(M_v)_{v\in V(H)}$ of $V(G)$ such that
\begin{enumerate}
\item for each $u\neq v\in V(H)$, $d_G(M_u, M_v)\geq k$;
\item for each $e=uv\in E(H)$ there is a path $P_e$ whose two endpoints lie in
  $M_u$ and $M_v$ and internal vertices are not in $\bigcup_{v\in
    V(H)}M_v$, and
  \item for every $e\neq e'\in E(H)$, $d_G(P_e, P_{e'})\geq k$ and for
    every $e=uv\in E(H)$ and $w\notin \sg{u,v}$, $d_G(P_e, M_w)\geq k$.
  \end{enumerate}
  A graph $H$ is an \emph{asymptotic minor} of $G$ if for every $k\geq 0$, $H$
is a $k$-fat minor of $G$.

\medskip

Let  $(X,d_X)$ and $(Y,d_Y)$ be two metric spaces. We say that $X$ is
\emph{quasi-isometric} to $Y$ if there is a map $f: X \rightarrow Y$
and constants $\varepsilon\ge 0$, $\lambda\ge 1$, and $C\ge 0$ such that
(i) for any $y\in Y$ there is $x\in X$ such that $d_Y(y,f(x))\le C$,
and (ii) for every $x_1,x_2\in X$, $$\frac1{\lambda}d_X(x_1,x_2)-\varepsilon\le d_Y(f(x_1),f(x_2))\le \lambda d_X(x_1,x_2)+\varepsilon.$$
It is not difficult to check that the definition is symmetric, and we
often simply say that $X$ and $Y$ are quasi-isometric. If 
condition (i) is omitted in the definition above, we say that $f$ is a
\emph{quasi-isometric embedding of $X$ in $Y$}.

\smallskip

We can view each graph $G$ as a metric space, by considering the natural shortest-path metric
associated to $G$.
A graph is \emph{locally finite} if every vertex has finite degree.
Georgakopoulos and Papasoglu conjectured the following \cite{GP23}.

\begin{conjecture}[Conjecture 9.3 in \cite{GP23}]\label{q1}
 If $G$ is locally finite, vertex-transitive and excludes some finite graph $H$ as an asymptotic minor, then $G$ is quasi-isometric to a planar graph.
\end{conjecture}

%Does Thomassen's approach for quasi-$4$-connected graphs adapt for
%$k$-fat minors?

It is natural to first prove this conjecture when $G$ excludes
some minor $H$ (instead of an \emph{asymptotic} minor, as defined above).
This suggests the following:
\begin{question}\label{q2}
  Is it true that if $G$ is locally finite, vertex-transitive and excludes some finite graph $H$ as a minor, then $G$ is quasi-isometric to a planar graph?
\end{question}

Our first result is a  positive answer to this question, in a slightly
stronger form. An infinite graph is \emph{quasi-transitive} if its
vertex set has finitely many orbits under the action of its
automorphism group.  Note that any vertex-transitive graph is
quasi-transitive, and that for quasi-transitive graphs, being locally
finite is equivalent to having bounded degree.

\smallskip

The \emph{countable clique}
$K_\infty$ is the graph with vertex set $\mathbb{N}$ in which every
two vertices are adjacent (every graph which excludes a finite or
countable graph $H$ as a minor also excludes $K_\infty$ as a
minor). We prove the following.

\begin{theorem}\label{thm:qp}
 Every locally finite quasi-transitive $K_{\infty}$-minor free
 graph is quasi-isometric to a planar graph of bounded degree.
\end{theorem}

The main technical tool that we use is a recent structural theorem on
locally finite quasi-transitive graphs excluding the countable
clique as a minor \cite{EGL23}, which shows that such graphs have a
canonical tree-decomposition in which all torsos are planar or finite
(see the next section for the definitions). Most importantly, this
result does not use the Robertson-Seymour graph minor structure
theorem.

We note that the result which allows us to construct the
quasi-isometry using the canonical tree-decomposition was also proved
recently (and 
independently) by MacManus \cite{Mac23} in a slightly different form (the ``if'' direction in his
Corollary C). Our proof is very similar to his.

\medskip

We now discuss several applications of Theorem \ref{thm:qp}.

%\ugo{Bounded degee est un invariant g\'eom\'etrique, donc peut-\^etre inutile de le pr\'eciser. Pour montrer qu'on a une quasi-isométrie vers un Cayley planaire, ça devrait découler de \cite[Theorem 7.4]{Mac23}.}

%A natural question is whether the planar graph of bounded degree in
%which the quasi-transitive graph embeds can be chosen to be
%quasi-transitive itself.

% Quasi-isometry to a quasi-transitive planar-graph seems harder to
% obtain, because of the following example: let $G$ be a graph
% constructing from a $K_5$ by attaching one triangular grid to each
% triangle of $K_5$, and then keep going adding a new $K_5$ in each
% triangle of $G$ and so on in a tree-way. It is not clear there how to
% get a quasi isometry to a quasi-transitive planar graph as some edge
% removal has to be done at some point, which does not seem canonical.

% \medskip

\subsection*{Application 1. Beyond planarity}

A graph is \emph{$k$-planar} if it has a drawing in the plane in which each
edge is involved in at most $k$ crossings (note that with this terminology,
being planar is the same as being  0-planar). The \emph{local crossing
  number} of a graph $G$, denoted by $\lcr(G)$, is the infimum integer
$k$ such that $G$ is $k$-planar.

\smallskip

Georgakopoulos and Papasoglu raised the following problem \cite{GP23}.

\begin{problem}[Problem 9.4 in \cite{GP23}]\label{p1}
For any quasi-transitive graph $G$ of bounded degree, $G$
      is quasi-isometric to a planar graph if and only if $G$ has
      finite local crossing number.
\end{problem}

%We prove that every graph which is quasi-isometric to
%a planar graph of bounded degree is $k$-planar for some $k$. As a
%consequence, we obtain our second application of Theorem
%\ref{thm:qp}:

We prove that for any integer $k$, every bounded degree graph which is quasi-isometric to
a $k$-planar graph is $k'$-planar for some integer
$k'$. In
the particular case $k=0$, we immediately obtain the ``only if''
direction of Problem \ref{p1} (we recently learned from Agelos
Georgakopoulos that he also proved the case $k=0$ independently). In Section \ref{sec:ccl}, we raise a number of conjectures whose validity would
imply a positive answer to the ``if'' direction of Problem
\ref{p1}. In the case $k=0$, 
we also obtain our second application of Theorem \ref{thm:qp}:

\begin{theorem}\label{thm:kpg}
Every locally finite quasi-transitive graph $G$ which is
$K_\infty$-minor-free has finite local crossing number.
\end{theorem}

The assumption that $G$ is locally finite is necessary, as shown by
the graph obtained from the square grid by adding a universal vertex
(this graph is $K_6$-minor free, but is not $k$-planar for any
$k<\infty$). The assumption that $G$ is quasi-transitive is also
crucial: consider for each integer $\ell$ a graph $G_\ell$
obtained from the square grid by adding an edge between  two vertices
at distance $\ell$ in the grid (if $G_\ell$ is $k$-planar then $k=\Omega(\ell)$),
and take the disjoint union of all graphs $G_\ell$, $\ell\in
\mathbb{N}$.

Note that in the other direction, there exist 1-planar graphs that are
vertex-transitive and locally finite, but which contain all graphs as
minors (the square grid with all diagonals is such an example).

\subsection*{Application 2. Assouad-Nagata dimension}

Our second application of Theorem \ref{thm:qp} requires the notions
of asymptotic dimension and Assouad-Nagata dimension of metric
spaces, which we introduce now. Let $(X,d)$ be a metric space, and let $\mathcal{U}$ be a family of subsets of
$X$. We say that $\mathcal{U}$ is \emph{$D$-bounded} if each set $U\in \mathcal{U}$ has diameter at
most $D$. We say that $\mathcal{U}$ is \emph{$r$-disjoint} if for any
$a,b$ belonging to different elements of $\mathcal{U}$ we have
$d(a,b)> r$.

\medskip

We say that $D_X:\mathbb{R}^+\to
\mathbb{R}^+$ is an
\emph{$n$-dimensional control function} for $(X,d)$ if for any $r>0$, $(X,d)$ has a cover
$\mathcal{U}=\bigcup_{i=1}^{n+1}\mathcal{U}_i$, such that each
$\mathcal{U}_i$ is $r$-disjoint and each element of $\mathcal{U}$ is $D_X(r)$-bounded. 
A control function $D_X$ for a metric space $X$ is said to be
a \emph{dilation} if there is a constant $c>0$ such that $D_X(r)\le
cr$, for any $r>0$.

\medskip

The \emph{asymptotic dimension} of $(X,d)$, introduced by Gromov in
\cite{Gro93},
% and denoted by $\mathrm{asdim}\, (X,d)$,
is the
least integer $n$ such that $(X,d)$ has an $n$-dimensional control function. If no such integer $n$ exists, then the asymptotic
dimension is infinite. 
The \emph{Assouad-Nagata dimension} of $(X,d)$, introduced by Assouad
in \cite{Ass82},
is the least $n$ such that $(X,d)$ has an $n$-dimensional control function which is a
dilation. Clearly the asymptotic dimension is at most the Assouad-Nagata
dimension. 

%\medskip

%We can view each graph $G$ as a metric space, by considering the natural shortest-path metric
%associated to $G$.

\medskip

It was first proved  that planar graphs have Assouad-Nagata dimension
at most 3 \cite{FP21}, later improved to 2 by \cite{BBEGLPS,JL22}. It
was also proved in \cite{BBEGLPS} that any graph excluding a minor has
asymptotic dimension at most 2. This was finally extended by Liu in
\cite{Liu23}, who proved that any graph avoiding a minor has
Assouad-Nagata dimension at most 2 (a different proof was then given by
Distel in \cite{Dis23}). All the results on graphs excluding a minor
mentioned above (even earlier results in \cite{BBEGPS} for bounded degree graphs)
crucially rely on the Graph minor structure theorem by Robertson and
Seymour \cite{RS-XVI}, a deep result proved in a series of 16 papers.

\medskip

Using the invariance of Assouad-Nagata dimension under bilipschitz
embedding \cite{LS05}, we will deduce from Theorem \ref{thm:qp} that minor-excluded
quasi-transitive graphs have Assouad-Nagata dimension at most 2. A
finitely generated group is said to be \emph{minor-excluded} if it has
a Cayley graph which excludes a minor. Our result
directly implies that finitely generated minor-excluded groups have
Assouad-Nagata dimension at most 2, and thus 
asymptotic dimension at most 2  (which was originally
conjectured by Ostrovskii and Rosenthal in \cite{OR}).
% We will only use Theorem
%\ref{thm:qp} and a few simple tools from \cite{BBEGPS,BBEGLPS} based on the
%work of Brodskiy, Dydak, Levin and Mitra \cite{BDLM08}. In particular
%we will give a short proof of the fact  that planar graphs of bounded degree have
%Assouad-Nagata dimension at most 2.
Crucially, our
proof \emph{does not rely} on the Graph minor structure theorem of Robertson and
Seymour.

\section{Preliminaries}

All graphs in this paper are assumed to be infinite, unless stated
otherwise.
%A graph is \emph{locally finite} if every vertex has finite
%degree.

\smallskip

A {\it tree-decomposition} of a graph $G$ is a pair $(T,\mathcal{V})$ such that $T$ is a tree and $\mathcal{V}$ is a collection $(V_t: t \in V(T))$ of subsets of $V(G)$, called the {\it bags}, such that
	\begin{itemize}
		\item $\bigcup_{t \in V(T)}V_t = V(G)$,
		\item for every $uv \in E(G)$, there exists $t \in
                  V(T)$ such that $u,v\in V_t$, and
		\item for every $v \in V(G)$, the set $\{t \in V(T): v \in V_t\}$ induces a connected subgraph of $T$.
                \end{itemize}

                \smallskip
                
For a tree-decomposition $(T,\mathcal{V})$, the {\it width} of
$(T,\mathcal{V})$ is $\sup_{t \in V(T)}\lvert V
_t \rvert-1\in \mathbb N\cup \sg{\infty}$.
The {\it treewidth} of $G$ is the minimum width of a
tree-decomposition of $G$.

\medskip

The sets $V_t\cap V_{t'}$ for every $tt'\in E(T)$ are called the
\emph{adhesion sets} of $(T,\mathcal V)$ and the \emph{adhesion} of
$(T,\mathcal V)$ is the supremum of the sizes of its adhesion sets (possibly
infinite). For $t\in V(T)$, the \emph{torso} $G\torso{V_t}$ is the
graph obtained from $G[V_t]$ (the subgraph of $G$ induced by the bag
$V_t$) by adding all edges $uv$ with $u,v\in V_t$ for which there exist $t'$
such that $u$ and $v$ lie in the adhesion set $V_t\cap V_{t'}$.

\medskip

We say that a
tree-decomposition $(T,\mathcal V)$ is \emph{canonical}, if for every
automorphism $\gamma$ of $G$, $\gamma$ sends bags of $(T,\mathcal V)$ to bags, and
adhesion sets to adhesion sets. In other words, the automorphism group
of $G$ induces a
group action on $T$.

\medskip

A \emph{separation} in a graph $G=(V,E)$ is a triple $\Sep$
such that $Y,S,Z$ are pairwise disjoint, $V=Y\cup S\cup Z$ and there is no edge
between vertices of $Y$ and $Z$. 
The separation $\Sep$ is said to be
\emph{tight} if there are some components $C_Y,C_Z$ respectively of $G[Y],G[Z]$
such that $N_G(C_Y)=N_G(C_Z)=S$.

\medskip

Consider a tree-decomposition $(T,\mathcal V)$ of a graph $G$, with
$\mathcal{V}=(V_t)_{t\in V(T)}$. Let $A$ be an orientation of the
edges of $E(T)$, i.e.\ a choice of either $(t_1,t_2)$ or $(t_2,t_1)$
for every edge $t_1t_2$ of $T$. For a pair $(t_1,t_2)\in
A$, and for each $i\in \{1,2\}$, let $T_i$ denote the component of
$T-\{t_1t_2\}$ containing $t_i$. Then the \emph{edge-separation} of
$G$ associated to $(t_1,t_2)$ is $(Y_1,S,Y_2)$ with $S:= V_{t_1}\cap
V_{t_2}$ and $Y_i:=
\bigcup_{s\in V(T_i)} V_s\setminus S$ for $i\in \sg{1,2}$.

\medskip

We will need the main result of \cite{EGL23}, which gives the
structure of locally finite quasi-transitive graph excluding a minor
(note that this result does not use the Graph minor structure theorem by Robertson and
Seymour \cite{RS-XVI}).

\begin{theorem}[\cite{EGL23}]
 \label{thm: mainCTTD}
 Let $G$ be a locally finite quasi-transitive graph excluding $\Kinf$ as a minor. Then there is an integer $k$ such that $G$ admits a canonical tree-decomposition $(T,\mathcal V)$ of adhesion at most $k$, whose torsos have size at most $k$ or are planar.
 Moreover, the edge-separations of $(T,\mathcal V)$ are tight.
\end{theorem}

% Since several results in our paper are related to graphs embedded in the
% plane with restrictions on how the edges cross, we mention an
% interesting similarity between Theorem \ref{thm:
%   mainCTTD}  and the structure of graphs excluding a graph with a single crossing as minor. It was proved by Robertson and
% Seymour \cite{RS93-1} that if a finite graph $H$ can be embedded in
% the plane with a single crossing, then there is a constant $k_H$ such
% that any finite graph excluding $H$ as a minor has a
% tree-decomposition with adhesion at most $k_H$ in which all torsos are
% planar or have size at most $k_H$.

% \medskip

The property that the edge-separations $(T,\mathcal V)$  are tight in
the statement of Theorem \ref{thm: mainCTTD} will be particularly 
useful in combination with the following result of Thomassen and Woess
\cite{TW} (which was explicitly proved for transitive graphs, but the
same proof also holds for quasi-transitive graphs).
\begin{lemma}[Corollary 4.3 in \cite{TW}]
\label{lem:TWcut}
 Let $G$ be a locally finite graph. Then for every $v\in V(G)$ and $k\geq 1$, there is 
 only a finite number of tight separations $\Sep$ of order $k$ in $G$ such that $v\in S$.
 Moreover, if $G$ is quasi-transitive then for any $k\geq 1$, there is only a
 finite number of orbits of tight separations of order at
 most $k$ in $G$ under the action of the automorphism group of $G$. 
\end{lemma}

\section{Proof of Theorem \ref{thm:qp}}

We assume that $G$ is connected, since otherwise we can consider each
connected component separately.
By Theorem \ref{thm: mainCTTD}, $G$ has a canonical tree-decomposition
 $(T,\mathcal V)$ whose torsos $G\torso{V_t}$, $t\in V(T)$, are either
 planar or finite and whose adhesion sets have bounded size. For each $u\in V(G)$, we let $T_u$ be the subtree of $T$ with vertex set $\sg{t\in V(T), u\in V_t}$. Note that as the edge-separations of $(T, \mathcal V)$ are tight, Lemma \ref{lem:TWcut} implies that $T_u$ is finite for each $u$.
 
%  Note that as the separations of $(T,\mathcal V)$ are
%  tight, by Lemma \ref{lem:TWcut}, for each vertex $u\in V(G)$ the subtree $T_u$ of $T$ with vertex set $\sg{t\in V(T), u\in V_t}$ is finite, and there is a uniform upper bound $N\geq 0$ on the sizes $|V(T_u)|$ that does not depend of $u\in V(T)$.
 
 We let $G'$ be the graph constructed as follow:
 for each $t\in V(T)$, we let $V'_t$ be a copy of $V_t$ and $G'_t$
 (with vertex set $V_t'$) be a copy of $G\torso{V_t}$ if $V_t$ is
 infinite, or a spanning tree of $G\torso{V_t}$  if $V_t$ is finite. For each $u\in V(G)$ and $t\in V(T_u)$, we let $u^{(t)}$ denote the copy of $u$ in $V'_t$.
 We let $V(G'):=\biguplus_{t\in V(T)}V'_t$. Now for every edge $st\in
 E(T)$, we choose an arbitrary vertex $u_{st}\in V_t\cap V_s$ (such a
 vertex exists, since $G$ is assumed to be connected). We let:
 $$E(G'):= \left(\biguplus_{t\in V(T)}E(G'_t)\right) \uplus \sg{u_{st}^{(s)}u_{st}^{(t)}, st\in E(T)}.$$

 We also let $T'$ be the 1-subdivision of $T$ (the graph obtained from
 $T$ by replacing each edge $e=st$ by a two-edge path $s,t_e,t$). Finally, for each $e=st\in E(T)$, we set $V'_{t_e}:=\sg{u_{st}^{(s)},u_{st}^{(t)}}$ and $\mathcal V':=(V'_t)_{t\in V(T')}$.
 We observe  that by definition,  $(T', \mathcal V')$ is a
 tree-decomposition of $G'$ whose adhesion sets all have size $1$. In
 particular, for every $t\in V(T')$, $G'[V_t']=G'\torso{V_t'}$. We
 also note that by the definition of $G'$ and Lemma \ref{lem:TWcut},
 $G'$ has bounded degree.

 \begin{claim}
  For every graph $G$, if $G$ has a tree-decomposition $(T,\mathcal V)$ such that every torso is planar and adhesion sets have size at most $1$, then $G$ is planar.
 \end{claim}

 \begin{proofofclaim}
 If $G$ contains  $K_5$ or $K_{3,3}$ as a minor, then some torso of
 $(T,\mathcal V)$ must also contain $K_5$ or $K_{3,3}$ as a minor, which
 is a contradiction. The result then follows from Wagner's theorem
 \cite{Wag37} stating that any graph excluding $K_5$ and $K_{3,3}$ is planar.
 \end{proofofclaim}
 
We now construct a quasi-isometry $f$ from $G$ to $G'$. For each $u\in V(G)$, we choose some $t_u\in V(T_u)$ and set $f(u):=u^{(t_u)}$. We also let
$A_1:=\max\sg{\mathrm{diam}_G(V_t), V_t ~\text{is finite}}$, $A_2:= \max(1,\max\sg{|V_t|, V_t ~\text{is finite}})$ and
$B:=\max\sg{\mathrm{diam}_T(T_u), u\in V(T)}$, which all exist by
Lemma \ref{lem:TWcut}, as the edge-separations of $(T,\mathcal V)$ are
tight. We note that for each $t\in V(T)$ such that $V_t$ is finite,
since $G'[V_t']=G'\torso{V_t'}$ is connected, its diameter  is at most $A_2$.

\smallskip

We first show the following:  
  
\begin{claim}
\label{clm: A3}
 There exists a constant $C\geq 0$ such that for each $t\in V(T), u,v\in V_t$: 
 $$d_{G}(u,v)\leq C\cdot d_{G\torso{V_t}}(u,v).$$
\end{claim}
\begin{proofofclaim}
By Lemma \ref{lem:TWcut}, $E(T)$ has finitely many orbits under the action of the
 automorphism group of $G$. Hence, up to automorphism there are only finitely many pairs $\sg{u,v}$ such that $u,v$ lie in a common adhesion set of $(T, \mathcal V)$. In particular, as $G$ is connected this means that the set of values $\sg{d_G(u,v), \exists st\in E(T), u,v\in V_s\cap V_t}$ admits a maximum $C$. The claim follows from this observation.
\end{proofofclaim}

We now show that there is a constant $\alpha>0$ such that for every
$u,v\in V(G)$ and every $f(u)f(v)$-path $P'$ in $G'$, there exists a
$uv$-path $P$ of size at most $\alpha\cdot |P'|$ in $G$. By taking
$P'$ to be a shortest path from $f(u)$ to $f(v)$ in $G'$, this will imply in particular that 
$d_G(u,v)\leq \alpha \cdot d_{G'}(f(u),f(v))$.

\begin{claim}
\label{clm: QI-sens-1}
For every $u,v\in V(G)$ and $t, s\in V(T)$ such that $u^{(t)}v^{(s)}\in E(G')$ we have
$$d_{G}(u,v)\leq \alpha:=\max(A_1,C).$$
\end{claim}

\begin{proofofclaim}
 Assume first that $s=t$. If $V_t$ is finite, then $d_G(u,v)\leq A_1$. 
 If $V_t$ is infinite we must have $uv\in E(G\torso{V_t})$, and thus
 $d_{G}(u, v)\leq C$ by Claim \ref{clm: A3}.
 
 Assume now that $s\neq t$. Then by definition of $G'$, we must have
 $st\in E(T)$ and $u=v$, and thus $d_G(u,v)=0$.
\end{proofofclaim}

We now show that there exists a constant $\beta>0$ such that for every
$u,v\in V(G)$ and every $uv$-path $P$ in $G$, there exists a
$f(u)f(v)$-path $P'$ of size at most $\beta|P|$ in $G'$. This directly
implies that $d_{G'}(f(u),f(v))\le \beta\cdot d_G(u,v)$.

\begin{claim}
\label{clm: QI-sens-2}
For every $u,v\in V(G)$ and $t, s\in V(T)$ such that $uv\in E(G)$, $u\in V_t$ and $v\in V_s$ we have:
$$d_{G'}(u^{(t)},v^{(s)})\leq \beta:=(4A_2+2)B+A_2.$$
\end{claim}

\begin{proofofclaim}
First note that if $V_t$ is finite, then for each $u,v\in V_t$ we have:
$$d_{G'}(u^{(t)},v^{(t)})=d_{G'[V_t']}(u^{(t)},v^{(t)})\leq A_2.$$
If $V_t$ is infinite, then for each $u,v\in V_t$ such that $uv\in
E(G)$, we have $u^{(t)}v^{(t)}\in E(G')$ and thus
$d_{G'}(u^{(t)},v^{(t)})\le 1$. Since $A_2\geq 1$, it follows that for each $t\in V(T)$ and $u,v\in V_t$ such that $uv\in E(G)$, we have
\begin{equation}
\label{eq1}
d_{G'}(u^{(t)},v^{(t)})\leq A_2.
\end{equation}

Now let $u\in V(G)$ and $s,t \in V(T_u)$. We let $(s=t_0, t_1, \ldots,
t=t_\ell)$ be the shortest $st$-path in $T$. Note that it is also a
path in $T_u$, hence $\ell\leq B$. Recall that in the construction of
$G'$, we have chosen for each edge $st\in E(T)$ a vertex $u_{st}\in
V_s\cap V_t$ and we have added an edge in $G'$ between
$u_{st}^{(s)}\in V_s'$ and $u_{st}^{(t)}\in V_t'$. 
For each $i\in [\ell]$, we write 
$x_i:=u_{t_{i-1}t_i}\in V_{t_{i-1}}\cap V_{t_i}$ for the sake of
readability. Note that for each $i\in [\ell]$, $x_i$ might be equal to
$u$ and that both $u$ and $x_i$ lie in the adhesion set
$V_{t_{i-1}}\cap V_{t_i}$. This implies that $u^{(t_i)}$ and
$x_i^{(t_i)}$ are adjacent in $G'$ if $V_{t_i}$ is infinite, and
$d_{G'}(u^{(t_i)},x_i^{(t_i)})\le A_2$ otherwise. So
$d_{G'}(u^{(t_i)},x_i^{(t_i)})\le A_2$ in both cases, and similarly
$d_{G'}(u^{(t_{i-1})},x_i^{(t_{i-1})})\le A_2$. It follows that for each $i\in [\ell]$, we have
$$d_{G'}(u^{(t_{i-1})}, u^{(t_{i})}) \leq d_{G'}(u^{(t_{i-1})},
x_{i}^{(t_{i-1})}) + d_{G'}(x_{i}^{(t_{i-1})}, x_{i}^{(t_{i})}) +
d_{G'}(x_i^{(t_{i})}, u^{(t_{i})}) \leq 2A_2+1.$$ 
This implies that  for every $u\in V(G)$ and $s,t\in V(T_u)$
\begin{equation}
\label{eq2} 
d_{G'}(u^{(s)}, u^{(t)})\leq (2A_2+1)B.
\end{equation}

To conclude the proof of the claim, let $uv \in E(G)$. As $(T, \mathcal V)$ is a tree-decomposition, there exists some $t\in V(T)$ such that $u,v\in V_t$. Then:
$$d_{G'}(f(u), f(v))\leq d_{G'}(u^{(t_u)}, u^{(t)})+ d_{G'}(u^{(t)}, v^{(t)})+ d_{G'}(v^{(t)}, v^{(t_v)}),$$
thus by inequalities (\ref{eq1}) and (\ref{eq2}) we obtain
$d_{G'}(f(u), f(v))\leq (4A_2+2)B + A_2$.
% It immediately follows that for all $u,v\in V(G)$:
% $$d_{G'}(f(u),f(v))\leq (\max(C,A)+ 2A(2C+1))d_G(u,v).$$
\end{proofofclaim}

% Now let $t\in V(T)$ be such that $V_t$ is infinite and consider $u,v\in V_t$. Then $G\torso{V'_t}$ is isomorphic to $G\torso{V_t}$ and by Claim \ref{clm: A3}:
% $$d_{G'}(u^{(t)},v^{(t)})\leq C\cdot d_{G}(u,v).$$ \louis{c'est dans
%   le mauvais sens par rapport au claim \ref{clm: A3} non ?}

To prove that $f$ is a quasi-isometry, it remains to prove that each $y\in V(G')$ is at bounded distance in $G'$ from $f(V(G))$. For this, let $y\in V(G')$ and $t\in V(T), u\in V(G)$ be such that $y=u^{(t)}$.
Then by inequality (\ref{eq2}), $d_{G'}(y, f(u))\leq (2A_2+1)B$ so $f$ is
indeed a quasi-isometry.
This concludes the proof of Theorem  \ref{thm:qp}. \hfill $\Box$

% \begin{remark} A natural question is whether the planar graph $G'$ of bounded degree in
% which the quasi-transitive graph $G$ embeds can be chosen to be
% quasi-transitive itself.
% In the previous proof, if one would like to obtain a quasi-transitive
% graph $G'$, one needs to construct the application $u\mapsto t_u$ in a
% canonical way, which does not seem possible in general, unless the
% action of the automorphism group of $G$  on its vertex set is free, which is unlikely to happen: for example the groups acting freely on trees are exactly the virtually free groups. However, we could have more chance if we consider one-ended planar graphs which seem to admit quasi-transitive free group actions by translations.
% \end{remark}

\section{Beyond planarity}

In this section we prove Theorem \ref{thm:kpg}. We first show that for
graphs  of bounded degree, having finite local crossing
number is preserved under quasi-isometry.

\begin{theorem}
\label{thm:kp}
Let $G$ be a
graph of bounded degree which is
quasi-isometric to a graph $H$ of finite local crossing number. Then
$G$ also has finite local crossing number.
\end{theorem}

Note that in general, the property of being locally finite, or even of
having countably many vertices is not preserved under
quasi-isometry. The next lemma will be useful to make sure that we can
restrict ourselves to 
locally finite graphs in the remainder of the proof.

\begin{lemma}
  \label{lem: loc-finite}
  Let $G$ be a graph of bounded degree which is quasi-isometric to a
  graph $H$. Then $G$ is quasi-isometric to a subgraph $H'$ of $H$ of
  bounded degree.
\end{lemma}

\begin{proof}
 We let $f: V(G)\to V(H)$ and $A\ge 1$ be such that for each $x, x' \in V(G)$:
$$\frac{1}{A}\cdot d_G(x,x') - A \leq d_H(f(x), f(x'))\leq A \cdot d_G(x,x')+A,$$
and such that the $A$-neighborhood of $f(V(G))$ covers $H$. We also let $\Delta\in \mathbb N$ denote the maximum degree of $G$.
Note that for each $xy\in E(G)$, there exists a $f(x)f(y)$-path
$P_{xy}$ in $H$ such that $|P_{xy}|\leq A\cdot 1+1=2A$. We let $H'$ be the subgraph of $H$ given by the union of all such paths $P_{xy}$. 

We first observe that $H'$ is quasi-isometric to $G$, and that $f$
gives the corresponding quasi-isometric embedding. Note that for every
$z\in V(H')$, by construction there must be some edge $xy\in E(G)$
such that $z\in P_{xy}$. In particular, $d_{H'}(z,f(x))\leq 2A$. Note
that by construction we clearly have $d_{H'}(f(x), f(y))\leq
2Ad_G(x,y)$ for each $x,y\in V(G)$, and as $d_{H'}(f(x),f(y))\geq
d_H(f(x),f(y))$, $f$  indeed induces a quasi-isometric embedding between $G$ and $H'$.

Now we show that $H'$ has bounded degree. Let $z\in V(H')$ and $xy\in E(G)$ such that $z\in V(P_{xy})$. Then $d_{H'}(z, f(x))\leq 2A$ so $X:=\sg{x\in V(G), z\in V(P_{xy})}$ has diameter at most $4A$ in $G$. Note that as $G$ degree at most $\Delta$, we have $|X|\leq \Delta^{4A}$. In particular it implies that $H'$ has degree at most $\Delta^{4A}$. 
\end{proof}

Given a graph $G$ and an integer $k\ge 1$, the \emph{$k$-th power} of
$G$, denoted by $G^k$, is the graph with the same vertex set as $G$ in
which two vertices are adjacent if and only if they are at distance at
most $k$ in $G$. The \emph{$k$-blow-up} of $G$, denoted by $G\boxtimes
K_k$, is the graph obtained from $G$ by replacing each vertex $u$ by a
copy $C_u$ of the complete graph $K_k$, and by adding all edges
between pairs $C_u,C_v$ if and only if $u$ and $v$ are adjacent in $G$
(so that each edge of $G$ is replaced by a complete bipartite graph
$K_{k,k}$ in $G\boxtimes
K_k$). Quasi-isometries of bounded degree graphs are related to graph powers
and blow-ups by the following lemma.

% Theorem \ref{thm:kp} is an immediate consequence of Lemmas
% \ref{lem: loc-finite}, \ref{lem:kp1} and \ref{lem:kp2} below. 

\begin{lemma}\label{lem:kp1}
Let $H$ be a graph, and let $G$ be a
graph of degree at most $\Delta\in \mathbb N$ which is quasi-isometric to $H$. Then there is an integer $k$ such that $G$ is a subgraph of $H^k\boxtimes K_k$.
\end{lemma}

\begin{proof}
We let $A\ge 1$ and $f: V(G)\to V(H)$ be such that for each $x, x' \in V(G)$:
$$\frac{1}{A}\cdot d_G(x,x') - A \leq d_H(f(x), f(x'))\leq A \cdot d_G(x,x')+A,$$
and such that the $A$-neighborhood of $f(V(G))$ covers $H$. Note that for each $x,x'\in V(G)$ such that $f(x)=f(x')=y$ we must have $d_G(x,x')\leq A^2$, hence
$$|f^{-1}(y)|\leq B:=\Delta^{A^2}$$
for every $y\in V(H)$. We now show that $G$ is a subgraph of $H':=H^{2A}\boxtimes K_{B}$, which implies the lemma for $k:=\max(2A, B)$.

As in the definition of a blow-up, for each $v\in V(H)$  we
denote by $C_v$ the associated clique of size $B$ in $H'$. For every $v\in V(H)$ we fix an arbitrary injection $g_v: f^{-1}(v)\to C_v$, and define an injective mapping $g: V(G)\to V(H')$ by letting $g(x):=g_{f(x)}(x)$ for each $x\in V(G)$. In other words every two vertices of $G$ having the same image $v$ by $f$ are sent by $g$ in the same clique $C_v$ in $H'$.
By construction $g$ is injective, so we just need to check that it
defines a graph homomorphism to conclude that $G$ is a subgraph of
$H'$. Let $xy\in E(G)$. Then $d_H(f(x), f(y))\leq 2A$ so in particular
every vertex in $V_{f(x)}$ is at distance at most $2A$ to every vertex
in $V_{f(y)}$ in $H'$. In particular, this means that $g(x)g(y)\in
E(H')$, as desired.
\end{proof}

The next observation will allow us to slightly simplify the statement of
Lemma~\ref{lem:kp1}.

\begin{observation}\label{obs:pqi}
 For every graph $H$ of bounded maximum degree, and
 every integer $k$, there
 exists a graph $G$ of bounded maximum degree such that
 $\lcr(G)=\lcr(H)$ and 
 $H^k\boxtimes K_k$ is a subgraph of $G^{k+2}$.
\end{observation}

\begin{proof}
  Let $G$ be the graph obtained from $H$ by attaching to each vertex $k$
 pendant vertices of degree $1$. Note that the graph
 $H^k\boxtimes K_k$ is a subgraph of $G^{k+2}$. To see this, one can bijectively map
 each clique $C_v$ of $H$ for $v\in V(H)$ to the $k$ pendant
 vertices we attached to $v$ in $G$, and observe that it gives an
 isomorphism between the graph induced by these vertices in $G^3$,
 and $H\boxtimes K_k$. Since adding pendant vertices does not change
 the local crossing number, $\lcr(G)=\lcr(H)$. Moreover $G$ has
 bounded maximum degree if and only if $H$ has bounded maximum degree. 
\end{proof}

We can now combine the results above to deduce the following  corollary.

\begin{corollary}\label{cor:pqi}
If a graph $G$ with bounded degree is quasi-isometric to a graph of
bounded local crossing number, then there exists a planar graph $H$ of
bounded maximum degree and an integer $k$, such that $G$ is a subgraph
of $H^k$.
\end{corollary}

\begin{proof}
By Lemmas \ref{lem: loc-finite} and \ref{lem:kp1}, there is a graph $H_1$ of bounded local crossing
number and maximum degree and an integer $\ell$ such that  $G$ is a
subgraph of $H_1^\ell\boxtimes K_\ell$. By Observation \ref{obs:pqi},
there is a graph $H_2$ of bounded local crossing
number and maximum degree such that  $G$ is a subgraph of
$H_2^{\ell+2}$. Observe that every $s$-planar graph $F_1$ is a subgraph of
$F_2^{s+1}$, where $F_2$ is the planar graph obtained from $F_1$ by
placing a new vertex at each crossing (and note that if $F_1 $ has bounded
degree, then $F_2$ also has bounded degree). It follows that there is a
planar graph $H$ of bounded degree and an integer $k$, such that $G$ is a subgraph
of $H^k$.
\end{proof}

We now prove that bounded powers of planar graphs of bounded degree
are $\ell$-planar for some $\ell$. This was proved for finite graphs
in \cite[Lemma 12]{DMW23}.

\begin{lemma}[\cite{DMW23}]
\label{lem:kp2fin}
 Let $H$ be a finite planar graph of maximum degree at most $\Delta$ and let
 $G$ be a subgraph of $H^k$, for some integer $k$. Then $G$ is
 $\ell$-planar, for $\ell:=2k(k+1)\Delta^k$.
\end{lemma}

However, we need a version of Lemma \ref{lem:kp2fin} for infinite locally finite
graphs. The first option is to simply follow the proof of \cite{DMW23}, which starts with a planar drawing of $G$, and adds for any path $P$
of length at most $k$, an edge between the endpoints of $P$, drawn in a
close neighborhood around $P$. However in the locally finite case this
approach requires that the original planar drawing has the property that every
edge has a small neighborhood which does not intersect any other
vertices or edges of the graph. Such a drawing always exists but it
requires a little bit of work. So instead, we chose to extend Lemma \ref{lem:kp2fin}  to infinite locally finite graphs using a simple compactness
argument.

\begin{lemma}\label{lem:compac}
 Let $G$ be a  locally finite graph.  If there is an integer $\ell$ such that all finite
induced subgraphs of $G$ are $\ell$-planar, then $G$ is also
$\ell$-planar.
\end{lemma}

\begin{proof}
We first observe
that any $\ell$-planar embedding of a graph $H$ can be described combinatorially, by considering the
planar graph $H^+$ obtained from $H$ by replacing all crossings by new
vertices. The corresponding planar embedding of $H^+$ can be completely
described (up to homeomorphism) by its rotation system (the clockwise
cyclic ordering of the neighbors around each vertex), and there are
only finitely many such rotation systems if $H$ (and thus $H^+$) is
finite.

\smallskip

We are now ready to prove the lemma. We can assume that $G$ is connected, since otherwise we can consider
each connected component independently. Since $G$ is locally finite and
connected, it is countable and we can write
$V(G)=\{v_1,v_2,\ldots\}$. We define a rooted tree $T$ as follows. The
root of $T$ is the unique $\ell$-planar embedding of $G[\{v_1\}]$, up to
homeomorphism. For every $k\ge 1$, and any $\ell$-planar embedding of
$G_k:=G[\{v_1,\ldots,v_k\}]$ we add a node in the tree and connect it to
the node corresponding to the resulting $\ell$-planar embedding of
$G_{k-1}=G_k-v_k$ (the embedding obtained by deleting
$v_k$ in the embedding of $G_k$). The resulting tree $T$ is infinite
(since every graph $G_k$ is $\ell$-planar by
assumption), and locally finite (since every graph
$G_k$ has only finitely many different $\ell$-planar
embeddings, up to homeomorphism). By K\"onig's infinity lemma
\cite{Kon27} (or by repeated applications of the pigeonhole
principle), $T$ has an infinite path starting at the root. This
infinite path corresponds to a sequence of $\ell$-planar embeddings of
$G_k$, $k\ge 0$, with the property that for every $k\ge 0$, the
$\ell$-planar embedding of $G_k$ can be obtained from the
$\ell$-planar embedding of $G_{k+1}$ by deleting $v_{k+1}$ (and all
edges incident to $v_{k+1}$). By taking the union of all the
$\ell$-planar embeddings of $G_k$, $k\ge 0$, we thus obtain an
$\ell$-planar embedding of $G$, as desired.
\end{proof}

We obtain the following as a direct consequence.

\begin{corollary}
\label{lem:kp2}
 Let $H$ be a locally finite planar graph of maximum degree at most $\Delta$ and let
 $G$ be a subgraph of $H^k$, for some integer $k$. Then $G$ is
 $\ell$-planar, for $\ell:=2k(k+1)\Delta^k$.
\end{corollary}

\begin{proof}
Let $X$ be a finite
subset of $V(G)\subseteq V(H)$ and for any pair $x,x'\in X$ with
$d_H(x,x')\le k$, consider a path $P_{x,x'}$ of length at most $k$
between $x$ and $x'$ in $H$.  Let $Y$ be the union of $X$ and the
vertex sets of all the paths $P_{x,x'}$ defined above. Then $H[Y]$ is
a finite planar graph, and $G[X]$, the finite subgraph of $G$ induced by $X$,
is a subgraph of $H[Y]^k$. By Lemma \ref{lem:kp2fin}, $G[X]$ is
$\ell$-planar with $\ell:=2k(k+1)\Delta^k$. Since this holds for any
finite set $X$, it follows from Lemma \ref{lem:compac} that $G$ itself
is $\ell$-planar, as desired.
\end{proof}

Theorem \ref{thm:kp} is now a direct consequence of Corollary \ref{cor:pqi} and Corollary \ref{lem:kp2}. 
By combining Theorems \ref{thm:qp} and \ref{thm:kp}, we then immediately
deduce Theorem \ref{thm:kpg}.

%\begin{corollary}\label{cor:kp}
%For every locally finite quasi-transitive graph $G$ which is
%$K_\infty$-minor-free, there exists an integer $k$ such that $G$ is $k$-planar.
% \end{corollary}

\section{Assouad-Nagata dimension of
  minor-excluded groups}\label{sec:AN}

For two graphs $G$ and $H$, we say that a function $f:
V(G)\to V(H)$ is a \emph{bilipschitz} mapping from $G$ to $H$ if there are
constants $c_1,c_2>0$ such that for any $x,y\in V(G)$, \[c_1\cdot d_G(x,y)\le
  d_H(f(x),f(y))\le c_2 \cdot d_G(x,y).\]

When such a function $f$ exists we also say that the graph $G$ has a \emph{bilipschitz
  embedding in $H$}.
%Since graphs are uniformly discrete metric spaces (in the sense
%that any two distinct elements lie at distance at least 1 apart), any quasi-isometric embedding of a graph
%$G$ in a graph $H$ is also a bilipschitz
%embedding of $G$ in $H$.
Since graphs are uniformly discrete metric spaces (in the sense that any two distinct elements lie at distance at least 1 apart), for any quasi-isometric
embedding of a bounded degree graph $G$ in a graph $H$, there is a
bilipschitz embedding of $G$ in a graph obtained from $H$ by adding,
for each vertex $v$ of $H$, a bounded number of vertices of degree 1
adjacent to $v$ (note that in the journal version of the paper \cite{EG24}, it is
incorrectly stated that with the same assumptions there is a
bilipschitz embedding of $G$ in $H$).
Hence, Theorem \ref{thm:qp} has the following immediate consequence.

\begin{corollary}\label{cor:qp}
 Every locally finite quasi-transitive $K_{\infty}$-minor free
 graph has a bilipschitz embedding  in a planar graph of bounded degree.
\end{corollary}

It was proved in \cite{BBEGLPS,JL22} that planar graphs have
Assouad-Nagata dimension at most 2. The proofs are similar and both
quite short: they
combine an original argument of Fujiwara and Papasoglu \cite{FP21}, proving that planar graphs have
Assouad-Nagata dimension at most 3, with a dimension reduction
technique introduced by Brodskiy, Dydak, Levin and Mitra
\cite{BDLM08}. Since Assouad-Nagata dimension is
invariant under bilipschitz embedding \cite{LS05}, we obtain the
following as a direct application.

\begin{corollary}
Every locally finite  quasi-transitive $K_\infty$-minor free
graph has Assouad-Nagata dimension at most 2.
\end{corollary}

Recall that
a stronger version (without the quasi-transitivity assumption) was proved
in \cite{Liu23,Dis23}, but the version above has a reasonably simple
proof that does not rely on the Robertson-Seymour graph minor
structure theorem.

\medskip

\subsection*{Important note} In the journal version of the paper
\cite{EG24}  we gave a simple proof that planar graphs of
bounded degree have Assouad-Nagata dimension at most 2. However our
argument was incorrect, as we were incorrectly using \cite[Theorem 4.3]{BBEGLPS}. More precisely, we stated in \cite[Theorem 5.4]{EG24} that whenever any graph of a
class $\mathcal{C}$ has a layering
such that any number of consecutive layers induces a class of 
Assouad-Nagata dimension at most $n$, then the class $\mathcal{C}$ has
Assouad-Nagata dimension at most $n+1$. This is false, as graphs of
bounded layered treewidth show (these graphs have unbounded
Assouad-Nagata dimension \cite{BBEGLPS}, while any bounded number of
layers induce a graph of bounded treewidth, and such graphs have Assouad-Nagata
dimension at most 1 \cite{BBEGLPS}). Indeed, to apply \cite[Theorem
4.3]{BBEGLPS}, it is
required that in addition, the $n$-dimensional control function is also
linear in the number of layers (which is not the case for graphs of
bounded layered treewidth).

We would like to apologize to everyone
for the inconvenience, and express our gratitude to
Robert Hickingbotham, who noticed the issue.

% \begin{proof}
% By Corollary \ref{cor:qp}, every locally finite  quasi-transitive $K_\infty$-minor free
% graph $G$ has a bilipschitz embedding in some planar graph $H$ of bounded degree, which
% has Assouad-Nagata dimension at most 2. Since Assouad-Nagata dimension is
% invariant under bilipschitz embedding \cite{LS05}, $G$ has Assouad-Nagata dimension at most 2. 
% \end{proof}

\section{Open problems}\label{sec:ccl}

We now discuss possible extensions or variants of Theorem
\ref{thm:qp}.

\subsection*{Quasi-isometries to quasi-transitive planar graphs}

It was proved by MacManus~\cite{Mac23} that if a finitely generated group has a
Cayley graph which is quasi-isometric to a planar graph, then it is
quasi-isometric to a planar Cayley graph.  In the same spirit, it is
natural to ask whether we can require the planar graph of bounded
degree in Theorem \ref{thm:qp} to be quasi-transitive, or even a
Cayley graph. Our current proof does not preserve symmetries, as we do a
number of non-canonical choices for the images of the vertices. As
remarked by a referee, the
stronger question above, whether we can require the planar graph in
Theorem \ref{thm:qp} to be a Cayley graph, is a special case of a
Problem of Woess \cite[Problem 1]{Woe91}, which asked whether every
transitive graph is quasi-isometric to a Cayley graph. This turned out to have a
negative answer \cite{EFW12} in general, but the question restricted
to (quasi-)transitive graphs excluding a minor might still have a
positive answer. 

\subsection*{A finite list of obstructions}

MacManus \cite{Mac23} recently gave a precise characterization of quasi-transitive
graphs which are quasi-isometric to a planar graph, in terms of the
existence of a canonical tree-decomposition sharing similarities with
that of Theorem \ref{thm: mainCTTD}. It would be interesting to also find a
characterization in terms of obstructions. For instance, it was
conjectured in \cite{GP23} that a graph is quasi-isometric to a planar
graph if and only if it does not contain $K_5$ or $K_{3,3}$ as an
asymptotic minor (see also Question \ref{q1} for a version of this
problem in the particular case of 
transitive graphs).

\smallskip

 Examples of quasi-transitive graphs that are not quasi-isometric
    to any planar graph include Cayley graphs of a group of Assouad-Nagata
    dimension at least 3, for instance any grid in
    dimension 3. This rules out any generalization of Theorem
    \ref{thm:qp} using classes of polynomial growth or expansion. This example also shows that we cannot extend Theorem
    \ref{thm:qp} to all families of bounded queue-number or stack-number.

\smallskip
    
    Here is
    perhaps a more interesting example. The \emph{strong product} $G\boxtimes H$ of two graphs $G$ and $H$
    has vertex set $V(G)\times V(H)$, and two distinct vertices $(u,x)$ and
    $(v,y)$ are adjacent if and only if ($u=v$ or $uv\in E(G)$) and ($x=y$ or
    $xy\in E(H)$).
Consider the strong product $T\boxtimes P$
    of the infinite binary tree $T$ and the infinite 2-way path
    $P$. This graph has
    Assouad-Nagata dimension 2. On the other hand, we observe that for
    any integer $k$, $T\boxtimes P$ contains the complete bipartite
    graph $K_{k,k}$ as an asymptotic minor. To see this, remark that
    $T$ contains an infinite $k$-claw (the graph obtained by gluing
    $k$ infinite 1-way paths at their starting vertex) as an asymptotic minor (obtained by
    contracting a
    subtree of $T$ with $k$ leaves into a single vertex, and pruning
    the additional branches). The strong product of this infinite
    $k$-claw with $P$ consists of $k$ copies of an infinite
    2-dimensional grid (restricted to the upper half-plane, say),
    glued on a common infinite path $\pi$. On this path we can select $k$
    vertices, arbitrarily far apart, and on each infinite grid we can
    select a single vertex, arbitrarily far from $\pi$, and connect it
    to the $k$ vertices of $\pi$ via disjoint paths. By taking a ball
    of sufficiently large radius around each of the $2k$ vertices, we
    obtain $K_{k,k}$ as an $r$-fat minor for arbitrarily large $r$,
    and thus $K_{k,k}$ as an asymptotic minor.  This is
    illustrated for $k=3$ in Figure \ref{fig:claw}. Since containing a
    graph $H$ as an asymptotic minor is invariant under
    quasi-isometry, any graph $G$ which is quasi-isometric to $T\boxtimes
    P$ also contains every $K_{k,k}$ as an asymptotic minor, and thus
    every finite graph  as a minor (in particular, $G$ cannot be
    planar).

    \begin{figure}[htb]
  \centering
  \includegraphics[scale=0.85]{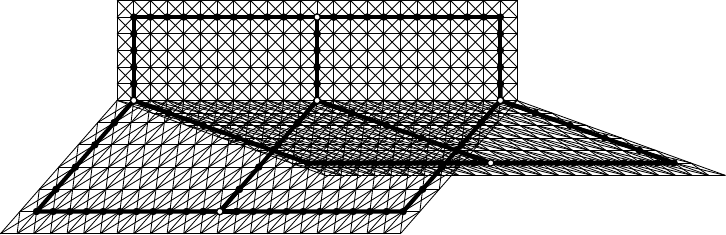}
  \caption{A fat $K_{3,3}$-minor in $T\boxtimes P$.}
  \label{fig:claw}
\end{figure}

    \smallskip

Recall that any graph excluding a minor has Assouad-Nagata dimension at
most 2 \cite{Liu23,Dis23}. It is natural to wonder whether some sort of
converse holds, that is whether any graph of Assoud-Nagata dimension
at most 2 is quasi-isometric to a graph excluding a minor (this would
be a natural extension of  Theorem \ref{thm:qp}). The example
above shows that this is false, even for vertex-transitive graphs.

\smallskip

As explained above, it was
conjectured in \cite{GP23} that graphs that are quasi-isometric to a planar graph
can be characterized by a finite list of forbidden asymptotic minors. A natural
question is whether this can be replaced by a finite list of forbidden
quasi-isometric graphs, at least in the case of quasi-transitive
graphs. We do not have a good candidate for such a finite list, but it
should contain at least the two examples mentioned above:
3-dimensional grids and the product of the binary tree with a path. One
difficulty is that no such list is even known (or conjectured to
exist) for Assouad-Nagata dimension at most 2, or asymptotic dimension
at most 2.

\subsection*{$k$-planar graphs}

We conjecture the following variant of Theorem \ref{thm:qp} for
$k$-planar graphs.

\begin{conjecture}\label{qikp}
Every quasi-transitive graph of bounded local crossing number and
degree is quasi-isometric to a planar graph.
\end{conjecture}

Using Theorem \ref{thm:kp}, Conjecture \ref{qikp} would imply
      a positive answer to Problem \ref{p1}.
We observe that it is enough to prove Conjecture \ref{qikp} for 1-planar
graphs.

\begin{conjecture}\label{qikp1}
Every quasi-transitive 1-planar graph of bounded
degree is quasi-isometric to a planar graph.
\end{conjecture}

To see that the case $\lcr\geq 2$ reduces to the case $\lcr=1$,
observe that
        for every $k$-planar graph $G$ with $k\geq 2$, its
        $(k-1)$-subdivision $G^{(k-1)}$ (the graph obtained from $G$
        by replacing every edge by a path on $k$ edges) is locally finite,
        quasi-transitive, quasi-isometric to $G$, and 1-planar. To see the last point,
        consider any embedding of $G$ in $\mathbb{R}^2$ in which every
        edge is involved in at most  $k$ crossings  and assume without loss
        of generality that the crossing points between every two edges
        are all pairwise distinct. Then for every edge $e\in E(G)$,
        one can add the $k-1$ corresponding vertices of $G^{(k-1)}$
        subdividing $e$ in the drawing by putting at least one vertex
        on each of the curves connecting two consecutive crossing
        points of $e$ with other edges.

        \smallskip

      We note that Conjecture \ref{qikp1} (and thus Conjecture
      \ref{qikp}) would be a direct consequence of the following.

      \begin{conjecture}
Let $G$ be a  quasi-transitive  1-planar graph of bounded
degree. Then there is an integer $k$ and an embedding
of $G$ in the plane with at most 1 crossing per edge such that for every pair of  crossing edges
$uv, xy$ in $G$, we have $d_G(u,x)\le k$.
\end{conjecture}

In a previous version of this manuscript we were conjecturing
something stronger, namely that for any embedding of $G$ with at most 1
crossing per edge, there is an integer $k$ such that all pairs of
crossing edges lie at distance at most $k$ in $G$. But this is false (as
shown by the two-way infinite path, drawn in such a way that it
self-intersects at more and more distant points).

\medskip

In this paper we have mainly considered graphs with finite local
crossing number. A natural generalization is the following: a graph is
\emph{$(<\omega)$-planar} if it has a drawing in
the plane in which each edge is involved in finitely many
crossings. This raises the following question.

\begin{question}\label{q:<w}
Let $G$ be a quasi-transitive graph of bounded degree which is
$(<\omega)$-planar. Is it true that $G$ has finite local crossing number?
\end{question}

It was observed by Kolja Knauer (personal communication) that Question
\ref{q:<w} has a negative answer, as indeed \emph{any} infinite
locally finite graph $G$
is $(<\omega)$-planar. To see this, consider an
ordering $v_1,v_2,\ldots$ of $V(G)$, map each vertex $v_i$ to the
point with coordinates $(i,i^2)$ in the
plane, and each edge $v_iv_j$ as a segment joining $v_i$ and
$v_j$. Note that by convexity of the function $x\mapsto x^2$, every edge crossing an edge $v_iv_j$ must have an endpoint
$v_k$ with $i< k < j$. As for every pair $i<j$ there are only finitely many such vertices
$v_k$ and each of them has finite degree, the edge $v_iv_j$ is crossed
by only finitely many other edges.

As there exist quasi-transitive graphs of bounded degree that have
unbounded local crossing number (the 3-dimensional grid for instance, see
\cite{DEW17}), the paragraph above implies that Question \ref{q:<w} has a negative answer.

 \subsection*{Acknowledgements}

  We thank Agelos Georgakopoulos and Joseph MacManus for their
  helpful comments and Kolja Knauer for allowing us to mention his
  negative answer to Question \ref{q:<w}. We also thank a referee for
  mentioning references \cite{EFW12} and \cite{Woe91} and for the
  helpful suggestions. Finally, we thank again Robert Hickingbotham
  for having spotted the issue reported at the end of Section \ref{sec:AN}.

\bibliographystyle{alpha}
\bibliography{biblio}

\end{document}